\documentclass{amsart}
\usepackage{epsfig,amssymb}

\theoremstyle{plain}
\newtheorem{theorem}{Theorem}[section]
\newtheorem{proposition}[theorem]{Proposition}
\newtheorem{lemma}[theorem]{Lemma}

\theoremstyle{definition}
\newtheorem*{definition}{Definition}

\theoremstyle{remark}
\newtheorem{remark}[theorem]{Remark}

\begin{document}

\title[Parity criterion and Dehn twists]%
 {Parity criterion and Dehn twists for \\unstabilized Heegaard splittings}

\author{Jung Hoon Lee}
\address{School of Mathematics, KIAS\\
207-43, Cheongnyangni 2-dong, Dongdaemun-gu\\
Seoul, Korea. Tel:\,+82-2-958-3736}
 \email{jhlee@kias.re.kr}

\subjclass[2000]{Primary 57M50} \keywords{Heegaard splitting,
parity condition, unstabilized}

\begin{abstract}
We give a parity condition of a Heegaard diagram to show that it
is unstabilized. This improves the result of \cite{Lee}. As an
application , we construct unstabilized Heegaard splittings by
Dehn twists on any given Heegaard splitting.
\end{abstract}

\maketitle

\section{Introduction}
For a closed $3$-manifold, a Heegaard splitting is a decomposition
of the manifold into two handlebodies. (For a $3$-manifold with
non-empty boundary, the manifold is decomposed into two
compression bodies along their common ``plus" boundary.)

The motivation of this paper started from tunnel number one knots.
Consider a tunnel number one knot $K$ in $S^3$ and an unknotting
tunnel $t$ for $K$. Consider two properly embedded arcs
$\gamma_1$, $\gamma_2$ in the exterior of $K$ which have nothing
to do with $t$. Suppose $K\cup \gamma_1\cup \gamma_2$ gives a
genus three Heegaard splitting of exterior of $K$. Is it
irreducible (or unstabilized)? In \cite{Kobayashi3}, Kobayashi
showed that every genus $g\ge 3$ Heegaard splitting of $2$-bridge
knot exterior is reducible. So the question is that whether there
exists an irreducible genus three Heegaard splitting of a tunnel
number one knot exterior which is not $2$-bridge. When a
non-minimal genus Heegaard splitting is given, in general it is
not an easy problem to show that it is irreducible or cannot be
destabilized.

However, there are infinitely many examples of manifolds having
non-minimal genus irreducible Heegaard splittings. In particular,
there exist $3$-manifolds having arbitrary high genus strongly
irreducible Heegaard splittings (\cite{CG}, \cite{Kobayashi1}).
 Casson and Gordon used the {\bf{rectangle condition}} on Heegaard
 diagrams to show strong irreducibility of such manifolds.
 (See (\cite{MS}, Appendix).)
 One can also refer to the papers \cite{Kobayashi2}, \cite{Lee},
 (\cite{Sedgwick}, section $7$),
 (\cite{Saito}, section $7$) for the rectangle condition.

 Rectangle condition is a condition on Heegaard diagrams for
 strong irreducibility. One can try to find a condition for
 irreducibility. Inspired by the example (torus)$\times S^1$, we
 gave a parity condition in \cite{Lee}, although it is not a weaker
 condition compared to rectangle condition. It is a condition
 on two collections of $3g-3$ essential disks giving pants
 decompositions of the Heegaard surface.

 We improve the parity condition of \cite{Lee}.
 It is known that a reducible Heegaard
 splitting of an irreducible manifold is stabilized.
 Hence, if the manifold under consideration is irreducible, the
 Heegaard splitting is irreducible. Figure 1. shows the relations
 of rectangle condition and parity condition for genus $g\ge 2$
 Heegaard splittings of irreducible manifolds.

 \begin{figure}[h]
   \centerline{\includegraphics[width=7cm]{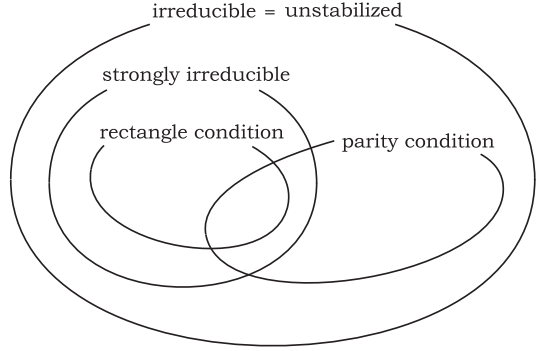}}
    \caption{Genus $g\ge 2 $ Heegaard splittings of irreducible manifolds}
\end{figure}

 \begin{theorem}
Let $M=H_1\cup_S H_2$ be a genus $g\ge 2$ Heegaard splitting of a
$3$-manifold $M$ and $\{D_1,D_2\cdots,D_g\}$ and
$\{E_1,E_2,\cdots,E_g\}$ be complete meridian disk systems of
$H_1$ and $H_2$, respectively.

If $|D_i\cap E_j|\equiv 0$ $(\bmod\,{2})$ for all the pairs
$(i,j)$, then $H_1\cup _S H_2$ is unstabilized.
\end{theorem}

As an application, in section $4$ we construct unstabilized
Heegaard splittings from any given splitting by doing a sequence
of Dehn twists.

\section{Planar decomposition and pants decomposition of a
surface}

 Let $H$ be a genus $g\ge 2$ handlebody and denote $\partial H$ by $S$.
 A collection of essential disks $\{D_1,D_2,\cdots,D_g\}$ in $H$
 is called a {\bf{complete meridian disk system}} for $H$ if
 the result of cutting $H$ along $\bigcup^{g}_{i=1} D_i$ is a $3$-ball.
 The corresponding result of cutting $S$ by
 $\bigcup^{g}_{i=1}\partial D_i$ is a planar surface, which is a
 $2g$-punctured sphere. We call it a {\bf{planar decomposition}}
 of $S$. This terminology was used in (\cite{Sedgwick}, section
 $7$).

 In another way, we can decompose $H$ and $S$ into smaller pieces
 with larger number of essential disks. Suppose a collection of
 mutually disjoint essential disks $\{D_1,D_2,\cdots,D_{3g-3}\}$
 cuts $H$ into $3$-balls $B_1,B_2,\cdots,B_{2g-2}$.
 We can imagine the shape of $B_i$ as a solid pair of pants.
 Let $P_i$ be the pair of pants $S\cap B_i$ ($i=1,2,\cdots,2g-2$).
 The decomposition $S=P_1\cup P_2\cup\cdots\cup P_{2g-2}$
 is called a {\bf{pants decomposition}} of $S$.

Let a planar decomposition of $S$ coming from a complete meridian
disk system $\{D_1,D_2,\cdots,D_g\}$ be given. We add $2g-3$ more
essential disks $\{\bar{D}_{g+1},\cdots,\bar{D}_{3g-3}\}$ of $H$
to the collection so that
$\mathcal{D}=\{D_1,D_2,\cdots,D_g,\bar{D}_{g+1},\cdots,\bar{D}_{3g-3}\}$
gives rise to a pants decomposition of $S$.
 Let $S=P_1\cup P_2\cup\cdots\cup P_{2g-2}$ be the new pants
 decomposition thus obtained.
We call $\bar{D}_i$ ($i=g+1,\cdots,3g-3$) as a {\bf{supplementary
essential disk}} for later use.

Give red color to $\partial D_i$ ($i=1,2,\cdots,g$) and blue color
to $\partial\bar{D}_j$ ($j=g+1,\cdots,3g-3$).

\begin{lemma}
For any $P_i$, if any two components among the three components of
$\partial P_i$ are red and the third is blue, convert the
blue-colored component also into red color. Iterate this operation
successively until it stops.

Then all curves constituting the pants decomposition of $S$ become
red colors.
\end{lemma}

\begin{figure}[h]
   \centerline{\includegraphics[width=10cm]{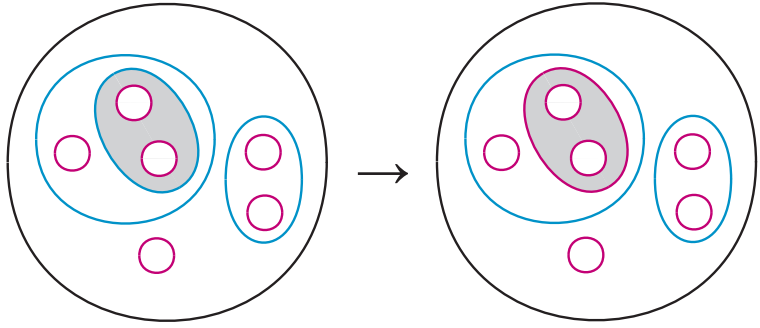}}
    \caption{The color is changed.}
\end{figure}

\begin{proof}
Originally the planar decomposition of $S$ gave a cutting of $S$
into a $2g$-punctured sphere with red boundaries. We added $2g-3$
essential blue loops on it to get a pants decomposition of $S$. So
there exists a pants having two red loops and one blue loop as its
boundary components. Then the color of the blue loop is changed to
red. See figure 2. In this way, an innermost blue loop co-bounds a
pants with two red loops, and it is changed into red color. Hence,
finally all curves come to have red colors.
\end{proof}

Let $\gamma$ be an essential simple closed curve in $S$. Note that
$\gamma$ can intersect $D_i$ or $\bar{D}_j$ only at the boundary
of the disk ($i=1,2,\cdots,g$ and $j=g+1,\cdots,3g-3$). We assume
that $\gamma$ intersects $(\bigcup^{g}_{i=1} D_i)\cup
(\bigcup^{3g-3}_{j=g+1}\bar{D}_{j})$ minimally. Let $|\cdot|$
denote the number of elements of a set.

\begin{lemma}
Suppose $|\gamma\cap D_i|\equiv 0$ $(\bmod\,{2})$ for all
$(i=1,2,\cdots,g)$. Then $|\gamma\cap \bar{D}_j|\equiv 0$
$(\bmod\,{2})$ for all $(j=g+1,\cdots,3g-3)$.
\end{lemma}

\begin{proof}
Suppose that $\gamma\cap ((\bigcup^{g}_{i=1} D_i)\cup
(\bigcup^{3g-3}_{j=g+1}\bar{D}_{j}))=\emptyset$. Then $\gamma$
lives in a pair of pants of the pants decomposition of $S$. Hence
either it is isotopic to $\partial D_i$ for some $i$ or
$\partial\bar{D}_j$ for some $j$. Then it is obvious that
$|\gamma\cap \bar{D}_j|\equiv 0$ ($\bmod\,{2}$) for all
$(j=g+1,\cdots,3g-3)$.

So we may assume that $\gamma$ intersects a pair of pants $P_k$ in
essential arcs. Let $\partial P_k$ be $l^1_k\cup l^2_k\cup l^3_k$.
If $|\gamma\cap l^1_k|$ and $|\gamma\cap l^2_k|$ are even numbers,
$|\gamma\cap l^3_k|$ should be an even number since $|\gamma\cap
(l^1_k\cup l^2_k\cup l^3_k)|$ should be an even number. (Every
properly embedded arcs in $P_k$ has two endpoints.)

Although the statements of Lemma 2.1 looks irrelevant with this
lemma, we use the idea of Lemma 2.1.
 In the proof of Lemma 2.1, every blue loop, which
was the boundary of $\bar{D}_j$, has eventually become a
red-colored loop because it co-bounded a pair of pants with two
other red loops. Since a red loop has even number of intersection
points with $\gamma$, we can see that $|\gamma\cap
\bar{D}_j|\equiv 0$ ($\bmod\,{2}$) by the conclusion of Lemma 2.1.
\end{proof}

Let $D$ be an essential disk in $H$. Since a handlebody is an
irreducible manifold, we may assume that $D\cap
((\bigcup^{g}_{i=1} D_i)\cup (\bigcup^{3g-3}_{j=g+1}\bar{D}_{j}))$
is a collection of arcs and the intersection is minimal. The
collection of arcs of intersection divides $D$ into subdisks. A
subdisk would be a $2n$-gon such as bigon, $4$-gon, $6$-gon, and
so on. Note that bigons are in one-to-one correspondence with
outermost disks in $D$. The following is a simple observation that
is important for the cut-and-connect operation that will be
discussed in section $3$.

\begin{lemma}
For all $i$ $(i=1,2,\cdots,g)$ and $j$ $(j=g+1,\cdots,3g-3)$,
$|\partial D\cap \partial D_i|\equiv 0$ $(\bmod\,{2})$ and
$|\partial D\cap \partial \bar{D}_j|\equiv 0$ $(\bmod\,{2})$.
\end{lemma}

\begin{proof}
Since any arc of intersection of $D\cap D_i$ has two endpoints,
$|\partial D\cap \partial D_i|$ would be an even number. The same
holds for $|\partial D\cap \partial\bar{D}_j|$.
\end{proof}

\section{Parity condition}

Let $H_1\cup_S H_2$ be a genus $g\ge 2$ Heegaard splitting of a
$3$-manifold $M$.  Let $\{D_1,D_2,\cdots,D_g\}$ and
$\{E_1,E_2,\cdots,E_g\}$
  be collections of complete meridian disk systems of $H_1$ and
  $H_2$,
  respectively. In \cite{Lee}, we gave a parity condition,
  involving two collections of $3g-3$ essential disks giving pants
  decompositions
  of both handlebodies of Heegaard splitting, to be unstabilized. Here we
  give a more improved condition for an unstabilized Heegaard
  splitting.

\begin{definition}
We say that $H_1\cup_S H_2$ satisfies the {\bf{even parity
condition}}
 if $|D_i\cap E_j|\equiv 0$ ($\bmod\,{2}$) for all the pairs $(i,j=1,2,\cdots,g)$.
\end{definition}

Assume that $H_1\cup_S H_2$ satisfies the even parity condition.
We add $2g-3$ more supplementary essential disks
$\{\bar{D}_{g+1},\cdots,\bar{D}_{3g-3}\}$ of $H_1$ to the
collection $\{D_1,D_2,\cdots,D_g\}$ so that
$\mathcal{D}=\{D_1,D_2,\cdots,D_g,\bar{D}_{g+1},\cdots,\bar{D}_{3g-3}\}$
gives rise to a pants decomposition of $S$. Also we add $2g-3$
more supplementary essential disks
$\{\bar{E}_{g+1},\cdots,\bar{E}_{3g-3}\}$ of $H_2$ to the
collection $\{E_1,E_2,\cdots,E_g\}$ so that
$\mathcal{E}=\{E_1,E_2,\cdots,E_g,\bar{E}_{g+1},\cdots,\bar{E}_{3g-3}\}$
gives rise to a pants decomposition of $S$. We assume that all the
boundaries of disks meet transversely and minimally.

 To simplify the notation, from now on we use the same subscript $i$ for $D_i$
 ($1\le i\le g$) and $\bar{D}_i$ ($g+1\le i \le 3g-3$). Also we
 use the same subscript $j$ for $E_j$ ($1\le j\le g$) and
 $\bar{E}_j$ ($g+1\le j\le 3g-3$).

 By applying the result of Lemma 2.2, we have the following.

 \begin{lemma}
 The parity of number of intersections are as follows.\\
\begin{enumerate}
\item[1)] For each $i$, $|D_i\cap \bar{E}_j|\equiv 0$
$(\bmod\,{2})$ for
all $j$. \\
\item[2)] For each $j$, $|\bar{D}_i\cap E_j|\equiv 0$
$(\bmod\,{2})$ for all
$i$.\\
\item[3)] For all $i$ and $j$, $|\bar{D}_i\cap \bar{E}_j|\equiv 0$
$(\bmod\,{2})$.
\end{enumerate}
\end{lemma}

\begin{proof}
1)\, For each $i$, from the definition of even parity condition,
$|D_i\cap E_j|\equiv 0$ ($\bmod\,{2}$) for all $j$. Then by Lemma
2.2, $|D_i\cap \bar{E}_j|\equiv 0$ ($\bmod\,{2}$) for all $j$.

2)\, For each $j$, from the definition of even parity condition,
$|D_i\cap E_j|\equiv 0\bmod\,{2}$ for all $i$. Then by Lemma 2.2,
$|\bar{D}_i\cap E_j|\equiv 0$ ($\bmod\,{2}$) for all $i$.

3)\, For each $i$, from 2) we can see that $|\bar{D}_i\cap
E_j|\equiv 0$ ($\bmod\,{2}$) for all $j$. Then by Lemma 2.2,
$|\bar{D}_i\cap \bar{E}_j|\equiv 0$ ($\bmod\,{2}$) for all $j$.
Then the result 3) follows. (This can be shown by using the result
1) also.)
\end{proof}

\begin{remark}
Lemma 3.1 means that any pair of disks from the collections
$\mathcal{D}$ and $\mathcal{E}$ have even number of intersections.
So Definition 1 implies (\cite{Lee}, Definition 4).
\end{remark}

Now we give the proof of Theorem 1.1.

\begin{proof}({\textit {of Theorem 1.1}})

\begin{figure}[h]
   \centerline{\includegraphics[width=10cm]{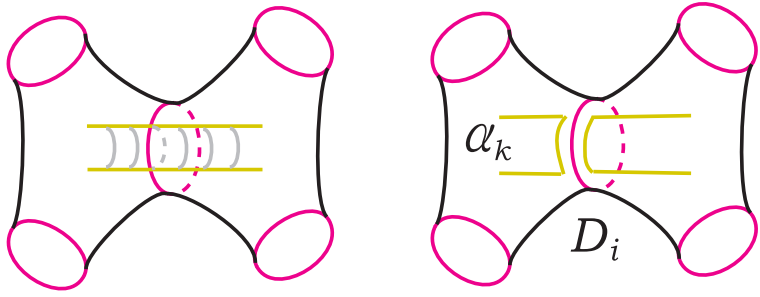}}
    \caption{$\alpha_k$ lives in a pair of pants.}
\end{figure}

Suppose that $H_1\cup_S H_2$ is stabilized. Then there exist
essential disks $D$ in $H_1$ and $E$ in $H_2$ such that $|D\cap
E|=1$. We may assume that the intersection $D\cap ((\bigcup
D_i)\cup(\bigcup \bar{D}_i))$ is a collection of arcs. Cut $D$ by
$(\bigcup D_i)\cup(\bigcup \bar{D}_i)$. Then $D$ is divided into
subdisks. For any arc, say $\gamma$, of intersection $D\cap D_i$
(or $D\cap \bar{D}_i$), two copies of $\gamma$, $\gamma_1$ and
$\gamma_2$ are created on both sides of $D_i$ (or $\bar{D}_i$)
which are parallel to each other. Connect two endpoints of
$\gamma_1$ and also connect two endpoints of $\gamma_2$ by arcs in
$S$ that are parallel and in opposite sides of $D_i$ (or
$\bar{D}_i$) to each other as in the Figure 3. We do this
cut-and-connect operation for all the arcs $D\cap ((\bigcup
D_i)\cup(\bigcup \bar{D}_i))$. Let $\{\alpha_k\}$ be the
collection of loops thence obtained from $\partial D$. Note that
each $\alpha_k$ lives in a pair of pants. Some $\alpha_k$ would be
isotopic to $\partial D_{i_k}$ and some other $\alpha_k$ be
isotopic to $\partial\bar{D}_{i_k}$ and some other $\alpha_k$
would possibly be a trivial loop.

Similarly, from $\partial E$ we obtain a collection of loops
$\{\beta_k\}$ by cut-and-connect operations. Some $\beta_k$ would
be isotopic to $\partial E_{j_k}$ and some other $\beta_k$ would
be isotopic to $\partial\bar{E}_{j_k}$ and some other $\beta_k$
would possibly be a trivial loop.

First we consider the parity of $|D\cap E_j|$ for each $j$ which
will be used in the below. Its parity is equivalent to $\sum_k
|\alpha_k\cap E_j|$ ($\bmod\,{2}$) since in the above
cut-and-connect operation two parallel copies $\gamma_1$ and
$\gamma_2$ were created. It is again equivalent to $\sum_k
|D_{i_k}\cap E_j|+\sum_k |\bar{D}_{i_k}\cap E_j|+\sum
|{\textrm{(trivial loop)}} \cap E_j|$ ($\bmod\,{2}$). By the even
parity condition and Lemma 3.1, it is even. Hence,
$$|D\cap E_j|\equiv 0\pmod{2}$$
By similar arguments, we have the following equalities in
($\bmod\,{2}$).
$$|D\cap \bar{E}_j|\equiv \sum_k |\alpha_k\cap \bar{E}_j|\equiv
\sum_k |D_{i_k}\cap \bar{E}_j|+\sum_k |\bar{D}_{i_k}\cap
\bar{E}_j|+\sum |{\textrm{(trivial loop)}} \cap \bar{E}_j|$$
 By Lemma 3.1, we have
$$|D\cap \bar{E}_j|\equiv 0 \pmod{2}$$
Now we have the following equalities in ($\bmod\,{2}$).
$$|D\cap E|
\equiv \sum_k |D\cap \beta_k| \equiv \sum_k |D\cap E_{j_k}|+
\sum_k|D\cap\bar{E}_{j_k}|+\sum |D\cap {\textrm{(trivial
loop)}}|$$
By above results, we have
$$ |D\cap E|\equiv 0 \pmod{2}$$
This is a contradiction since $|D\cap E|=1$. So we conclude that
$H_1\cup_S H_2$ is unstabilized.
\end{proof}

We give some examples of manifolds admitting a Heegaard splitting
satisfying the even parity condition.

\subsection{Connected sum of $S^2\times S^1$}

\begin{figure}[h]
   \centerline{\includegraphics[width=6cm]{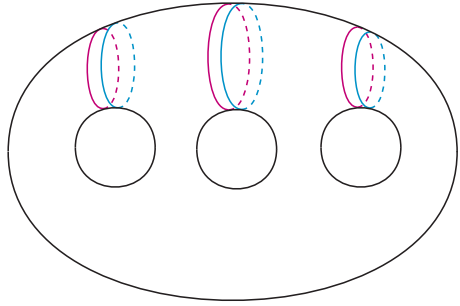}}
    \caption{Connected sum of three copies of $S^2\times
    S^1$, where
    $D_i\cap E_j=\emptyset$ for all $i$, $j$}
\end{figure}

 A connected sum of copies of $S^2\times S^1$ has a
 Heegaard splitting where each pairs of essential disks in both handlebodies are
 disjoint (Figure 4). This is a reducible and unstabilized Heegaard
 splitting.

\subsection{$($ Torus $)\times S^1$}

As an example of irreducible and unstabilized Heegaard splitting
satisfying the even parity condition, we consider a genus three
Heegaard splitting of (torus)$\times S^1$. Heegaard splittings of
manifolds of the form, (surface)$\times S^1$, are classified in
\cite{Schultens}.

\begin{figure}[h]
   \centerline{\includegraphics[width=9cm]{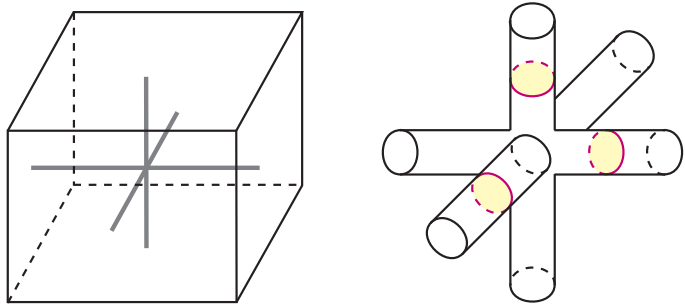}}
    \caption{Genus three handlebody $H_1$ in $({\textrm{torus}}\,)\times S^1$}
\end{figure}

One way of understanding a Heegaard splitting of (torus)$\times
S^1$ is as follows. Since (torus)$\times S^1$ is homeomorphic to
$S^1\times S^1\times S^1$, it can be obtained from a cube by
identifying three pairs of opposite faces. Consider the center of
the cube and center of each face. Connect the center of the cube
with the center of each face by an arc (Figure 5). Take a
neighborhood of it and after the identification of opposite
sectional disks, we get a genus three handlebody $H_1$. Figure 5.
shows a meridian disk system of $H_1$.

\begin{figure}[h]
   \centerline{\includegraphics[width=9cm]{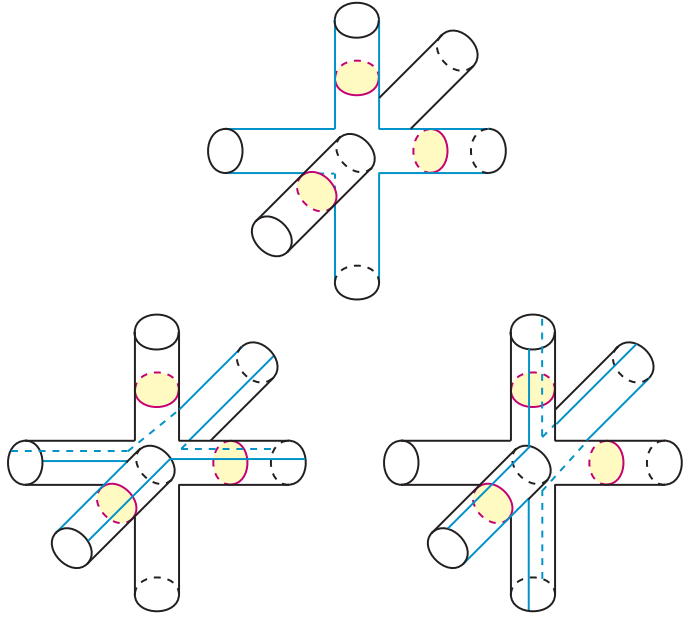}}
    \caption{Boundaries of essential disks of $H_2$ in $\partial H_1$}
\end{figure}

Now $H_2={\textrm{cl}}(H^c_1)$ is also a genus three handlebody.
Figure 6. shows the boundaries of essential disks of $H_2$ in
$\partial H_1$. We can see that it satisfies the even parity
condition. We can also see that it is weakly reducible. So it is
an irreducible and weakly reducible Heegaard splitting.

\section{Dehn twist}

In this section, we construct unstabilized Heegaard splittings by
Dehn twists from a given Heegaard splitting. First we examine the
parity of number of intersections of simple closed curves on a
surface after a Dehn twist. Let $T_{\gamma}:S\rightarrow S$ be a
homeomorphism of a closed surface $S$, which is a Dehn twist of
$S$ along $\gamma$.

\begin{lemma}
For essential simple closed curves $\gamma_1, \gamma_2, \gamma_3$,
we have
$$|T_{\gamma_1}(\gamma_2)\cap \gamma_3|\equiv |\gamma_2\cap
\gamma_3|+|\gamma_1\cap\gamma_2|\cdot|\gamma_1\cap\gamma_3|
\pmod{2}$$
\end{lemma}
\begin{proof}
Before Dehn twist, $\gamma_2$ and $\gamma_3$ have $|\gamma_2\cap
\gamma_3|$ number of intersection points. After the Dehn twist
$T_{\gamma_1}$, the number of intersection points is increased by
$|\gamma_1\cap\gamma_2|\cdot|\gamma_1\cap\gamma_3|$. Since
$T_{\gamma_1}(\gamma_2)$ and $\gamma_3$ can possibly have
inessential intersections (bigons), $|T_{\gamma_1}(\gamma_2)\cap
\gamma_3|$ is equivalent to $|\gamma_2\cap
\gamma_3|+|\gamma_1\cap\gamma_2|\cdot|\gamma_1\cap\gamma_3|$ by
($\bmod\,{2}$).
\end{proof}

By an application of Lemma 4.1, we make unstabilized Heegaard
splittings by a single Dehn twist from a given Heegaard splitting
satisfying the even parity condition.

\begin{proposition}
Suppose $H_1\cup_S H_2$ is an unstabilized Heegaard splitting
satisfying the even parity condition with $\mathcal{D}=\{D_1,
D_2,\cdots,D_g\}$ and $\mathcal{E}=\{E_1,E_2,\cdots,E_g\}$. Let
$\gamma$ be an essential simple closed curve in $S$ such that
$|\gamma\cap E_j|\equiv 0$ $(\bmod\,{2})$ for all $j$. Alter
$\mathcal{D}$ to $\mathcal{D'}=\{D'_1,D'_2,\cdots,D'_g\}$ by a
Dehn twist $T_{\gamma}$ and leave $\mathcal{E}$ unchanged.

Then the new Heegaard splitting $H'_1\cup_{S'} H'_2$  satisfies
the even parity condition, hence unstabilized.
\end{proposition}

\begin{proof}
We check the even parity condition for $H'_1\cup_{S'}H'_2$ by
using Lemma 4.1. Note that $\partial D'_i=T_{\gamma}(\partial
D_i)$. By Lemma 4.1,
$$|D'_i\cap E_j|\equiv|D_i\cap E_j|+|\gamma\cap
D_i|\cdot|\gamma\cap E_j| \pmod{2}$$
 In the above equation, $|D_i\cap E_j|$ is even by the even parity
 condition and $|\gamma\cap E_j|$ is even by the hypothesis of
 proposition. So $|D'_i\cap E_j|$ is even for all $i$ and $j$.
\end{proof}

\begin{remark}
Consider a neighborhood $N(\gamma)$ of $\gamma$ in $H_1$ such that
$N(\gamma)\cap {\textrm{cl}}(H_1-N(\gamma))$ is an annuls whose
core is parallel to $\gamma$. The Dehn twist $T_{\gamma}$ of $S$
is equivalent to removing $N(\gamma)$ from $H_1$ and attaching a
solid torus back so that a meridian of the attaching solid torus
is mapped to $\frac{1}{1}$-slope of $\partial N(\gamma)$ (A
longitude of the attaching solid torus is mapped to longitude, a
parallel of $\gamma$.) So in proposition 4.2, the ambient manifold
$M$ is changed to a new manifold $M'$ obtained by
$\frac{1}{1}$-Dehn filling on $\gamma$ in $M$.
\end{remark}

Now we are going to get an unstabilized Heegaard splitting
satisfying the even parity condition by Dehn twists from a
Heegaard splitting which does not satisfy the even parity
condition.

Suppose $|D_i\cap E_j|$ is odd for some $i$ and $j$. For the
convenience, assume that $|D_1\cap E_1|$ is odd.
 We consider the simple closed curve $\gamma=T_{\partial E_1}(\partial
 D_1)$ obtained by twisting  $\partial D_1$ along $\partial
 E_1$. First we examine the intersection of $\gamma$ with $D_i$
 and $E_j$.

\begin{lemma}
The parity of number of intersections of $\gamma$ with $D_i$ and
$E_j$ are as follows $(\bmod\,{2})$.\\

\begin{itemize}
\item[$\bullet$] $|\gamma\cap D_1|\equiv {\textrm{odd}}$\\
\item[$\bullet$] $|\gamma\cap D_i|\equiv
{\textrm{odd}}\cdot|D_i\cap E_1|$\,\, $(i=2,\cdots,g)$\\
\item[$\bullet$] $|\gamma\cap E_1|\equiv {\textrm{odd}}$\\
\item[$\bullet$] $|\gamma\cap E_j|\equiv |D_1\cap E_j|$\,\,
$(j=2,\cdots,g)$
\end{itemize}
\end{lemma}

\begin{proof}
$\bullet$\, By Lemma 4.1 and assumption, $|\gamma\cap
D_1|=|T_{\partial E_1}(\partial D_1)\cap D_1|\equiv|D_1\cap
D_1|+|E_1\cap D_1|\cdot |E_1\cap D_1|\equiv {\textrm{odd}}$
($\bmod\,{2}$).

$\bullet$\, Since $D_1$ and $D_i$ are disjoint, $|\gamma\cap
D_i|=|T_{\partial E_1}(\partial D_1)\cap D_i|\equiv |D_1\cap
D_i|+|E_1\cap D_1|\cdot|E_1\cap
D_i|\equiv{\textrm{odd}}\cdot|D_i\cap E_1|$ ($\bmod\,{2}$)\,
($i=2,\cdots,g$).

$\bullet$\, By assumption, $|\gamma\cap E_1|=|T_{\partial
E_1}(\partial D_1)\cap E_1|\equiv|D_1\cap E_1|+|E_1\cap
D_1|\cdot|E_1\cap E_1|\equiv {\textrm{odd}}$ ($\bmod\,{2}$).

$\bullet$\, Since $E_1$ and $E_j$ are disjoint, $|\gamma\cap
E_j|=|T_{\partial E_1}(\partial D_1)\cap E_j|\equiv|D_1\cap
E_j|+|E_1\cap D_1|\cdot|E_1\cap E_j|\equiv|D_1\cap E_j|$
($\bmod\,{2}$)\, ($j=2,\cdots,g$).
\end{proof}

Now we consider the Dehn twist $T_{\gamma}$ of $S$. Consider the
images of $\partial D_1$ and $\partial D_i$ ($i=2,\cdots,g$) after
the Dehn twist $T_{\gamma}$. We examine intersections of
$T_{\gamma}(\partial D_1)$ and $T_{\gamma}(\partial D_i)$
($i=2,\cdots,g$) with $E_1$ and $E_j$ ($j=2,\cdots,g$).

\begin{lemma}
The parity of number of intersections of $T_{\gamma}(\partial
D_1)$ and $T_{\gamma}(\partial D_i)$ ($i=2,\cdots,g$) with $E_1$
and $E_j$
($j=2,\cdots,g$) are as follows $(\bmod\,{2})$.\\

\begin{itemize}
\item[$\bullet$] $|T_{\gamma}(\partial D_1)\cap E_1|\equiv {\textrm{even}}$\\
\item[$\bullet$] $|T_{\gamma}(\partial D_1)\cap E_j|\equiv
{\textrm{even}}$\,\, $(j=2,\cdots,g)$\\
\item[$\bullet$] $|T_{\gamma}(\partial D_i)\cap E_1|\equiv {\textrm{even}}$\,\,
$(i=2,\cdots,g)$\\
\item[$\bullet$] $|T_{\gamma}(\partial D_i)\cap E_j|\equiv
|D_i\cap E_j|+{\textrm{odd}}\cdot|D_i\cap E_1|\cdot|D_1\cap
E_j|$\,\, $(i=2,\cdots,g)$\, $(j=2,\cdots,g)$
\end{itemize}
\end{lemma}

\begin{proof}
$\bullet$\, By Lemma 4.1 and Lemma 4.4 and assumption,
$|T_{\gamma}(\partial D_1)\cap E_1|\equiv |D_1\cap
E_1|+|\gamma\cap D_1|\cdot|\gamma\cap
E_1|\equiv{\textrm{odd}}+{\textrm{odd}}\cdot{\textrm{odd}}\equiv{\textrm{even}}$
($\bmod\,{2}$).

$\bullet$\, By Lemma 4.1 and Lemma 4.4, $|T_{\gamma}(\partial
D_1)\cap E_j|\equiv|D_1\cap E_j|+|\gamma\cap D_1|\cdot|\gamma\cap
E_j|\equiv|D_1\cap E_j|+{\textrm{odd}}\cdot|D_1\cap
E_j|\equiv{\textrm{even}}\cdot|D_1\cap E_j|\equiv{\textrm{even}}$
($\bmod\,{2}$)\, ($j=2,\cdots,g$).

$\bullet$\, By Lemma 4.1 and Lemma 4.4, $|T_{\gamma}(\partial
D_i)\cap E_1|\equiv|D_i\cap E_1|+|\gamma\cap D_i|\cdot|\gamma\cap
E_1|\equiv|D_i\cap E_1|+{\textrm{odd}}\cdot|D_i\cap
E_1|\cdot{\textrm{odd}}\equiv{\textrm{even}}\cdot|D_i\cap
E_1|\equiv{\textrm{even}}$ ($\bmod\,{2}$)\, ($i=2,\cdots,g$).

$\bullet$\, By Lemma 4.1 and Lemma 4.4, $|T_{\gamma}(\partial
D_i)\cap E_j|\equiv|D_i\cap E_j|+|\gamma\cap D_i|\cdot|\gamma\cap
E_j|\equiv|D_i\cap E_j|+{\textrm{odd}}\cdot|D_i\cap
E_1|\cdot|D_1\cap E_j|$ ($\bmod\,{2}$)\, ($i=2,\cdots,g$)\,
($j=2,\cdots,g$).
\end{proof}

Note that $|T_{\gamma}(\partial D_i)\cap E_1|$ and
$|T_{\gamma}(\partial D_1)\cap E_j|$ are even for all
$i,j=1,2,\cdots,g$ and the parity of difference
$|T_{\gamma}(\partial D_i)\cap E_j|-|D_i\cap E_j|$ is equivalent
to the parity of $|D_i\cap E_1|\cdot|D_1\cap E_j|$ for
$i,j=2,\cdots,g$. Let $\mathcal{D}'=\{D'_1,D'_2,\cdots,D'_g\}$ be
the collection of essential disks after the Dehn twist
$T_{\gamma}$ satisfying $\partial D'_i=T_{\gamma}(\partial D_i)$.
Suppose $|D'_i\cap E_j|$ is odd for some $i$ and $j$. Without loss
of generality, we may assume that $|D'_2\cap E_2|$ is odd. Let
$\gamma'=T_{\partial E_2}(\partial D'_2)$. Again we do a Dehn
twist $T_{\gamma'}$ of the surface and examine the parities
$|T_{\gamma'}(\partial D'_i)\cap E_j|$. By a sequence of Dehn
twists in this way, we can get a Heegaard splitting satisfying the
even parity condition as the following.

\begin{theorem}
For any given genus $g$ Heegaard splitting $H_1\cup_S H_2$ and
collections of complete meridian disk systems $\mathcal{D}=\{D_1,
D_2,\cdots,D_g\}$ and $\mathcal{E}=\{E_1,E_2,\cdots,E_g\}$, we get
an unstabilized Heegaard splitting satisfying the even parity
condition after a sequence of at most $g$ Dehn twists.

More precisely, the sequence of Dehn twists is $T_{\gamma_1},
T_{\gamma_2}, \cdots, T_{\gamma_g}$, where\\

\begin{itemize}
\item[$\bullet$] $\gamma_1=T_{\partial E_{j_1}}(\partial
D_{i_1})$ for some $i_1$ and $j_1$ with $|D_{i_1}\cap E_{j_1}|\equiv$ odd\\
\item[$\bullet$] $\gamma_2=T_{\partial
E_{j_2}}(T_{\gamma_1}(\partial D_{i_2}))$ for some $i_2$ and $j_2$ with
$|T_{\gamma_1}(\partial D_{i_2})\cap E_{j_2}|\equiv$ odd\\
\item[$\bullet$] $\gamma_{k}=T_{\partial E_{j_k}}\circ
T_{\gamma_{k-1}}\circ\cdots\circ T_{\gamma_1}(\partial D_{i_k})$
for some $i_k$ and $j_k$ with $|(T_{\gamma_{k-1}}\circ\cdots\circ
T_{\gamma_1}(\partial D_{i_k}))\cap E_{j_k}|\equiv$ odd\, $(k\le
g)$
\end{itemize}
\end{theorem}

\begin{proof}
As before, let $\gamma_1=T_{\partial E_1}(\partial D_1)$ without
loss of generality. Note that $|T_{\gamma_1}(\partial D_i)\cap
E_1|$ and $|T_{\gamma_1}(\partial D_1)\cap E_j|$ are even for all
$i,j=1,2,\cdots,g$ by Lemma 4.5. Let
$\mathcal{D}'=\{D'_1,D'_2,\cdots,D'_g\}$ be the collection of
essential disks after the Dehn twist $T_{\gamma_1}$ satisfying
$\partial D'_i=T_{\gamma_1}(\partial D_i)$. Let
$\gamma_2=T_{\partial E_2}(\partial D'_2)$ without loss of
generality. We examine the parity of $|T_{\gamma_2}(\partial
D'_i)\cap E_j|$. Let the notation
$\{1,2,\cdots,\hat{i},\cdots,g\}$ mean the set
$\{1,2,\cdots,i-1,i+1,\cdots,g\}$. By Lemma 4.5, we have the
following.\\

\begin{itemize}
\item[$\bullet$] $|T_{\gamma_2}(\partial D'_2)\cap E_2|\equiv {\textrm{even}}$\\
\item[$\bullet$] $|T_{\gamma_2}(\partial D'_2)\cap E_j|\equiv
{\textrm{even}}$\,\, $(j=1,\hat{2},\cdots,g)$\\
\item[$\bullet$] $|T_{\gamma_2}(\partial D'_i)\cap E_2|\equiv
{\textrm{even}}$\,\,
$(i=1,\hat{2},\cdots,g)$\\
\item[$\bullet$] $|T_{\gamma_2}(\partial D'_i)\cap E_j|\equiv
|D'_i\cap E_j|+{\textrm{odd}}\cdot|D'_i\cap E_2|\cdot|D'_2\cap
E_j|$\,\, $(i=1,\hat{2},\cdots,g)$\, $(j=1,\hat{2},\cdots,g)$\\
\end{itemize}

However, $|T_{\gamma_2}(\partial D'_1)\cap E_j|\equiv|D'_1\cap
E_j|+{\textrm{odd}}\cdot|D'_1\cap E_2|\cdot|D'_2\cap E_j|$ is
equal to $|T_{\gamma_1}(\partial D_1)\cap
E_j|+{\textrm{odd}}\cdot|T_{\gamma_1}(\partial D_1)\cap
E_2|\cdot|T_{\gamma_1}(\partial D_2)\cap E_j|$ and it is even
because $|T_{\gamma_1}(\partial D_1)\cap E_j|$ and
$|T_{\gamma_1}(\partial D_1)\cap E_2|$ are even by Lemma 4.5
again.

Also $|T_{\gamma_2}(\partial D'_i)\cap E_1|\equiv|D'_i\cap
E_1|+{\textrm{odd}}\cdot|D'_i\cap E_2|\cdot|D'_2\cap E_1|$ is
equal to $|T_{\gamma_1}(\partial D_i)\cap
E_1|+{\textrm{odd}}\cdot|T_{\gamma_1}(\partial D_i)\cap
E_2|\cdot|T_{\gamma_1}(\partial D_2)\cap E_1|$ and it is even
because $|T_{\gamma_1}(\partial D_i)\cap E_1|$ and
$|T_{\gamma_1}(\partial D_2)\cap E_1|$ are even by Lemma 4.5.

Hence we can see that $|T_{\gamma_2}(\partial D'_1)\cap E_j|$,
$|T_{\gamma_2}(\partial D'_2)\cap E_j|$, $|T_{\gamma_2}(\partial
D'_i)\cap E_1|$ and $|T_{\gamma_2}(\partial D'_i)\cap E_2|$ are
even for all $i$ and $j$. In this way, as we do sequence of Dehn
twists $T_{\gamma_k}$, the set of indices of even parity gets
bigger and bigger. So finally we get an unstabilized Heegaard
splitting satisfying the even parity condition after the sequence
of Dehn twists $T_{\gamma_1}, T_{\gamma_2}, \cdots, T_{\gamma_g}$.
\end{proof}

%


\begin{thebibliography}{99}

    \bibitem{CG}
    A. Casson \and C. Gordon,
    {\it Manifolds with irreducible Heegaard splittings of
    arbitrary large genus},
    Unpublished.

    \bibitem{Kobayashi1}
    T. Kobayashi,
    {\it A construction of $3$-manifolds whose homeomorphism
    classes of Heegaard splittings have polynomial growth},
    Osaka J. Math. {\bf 29} (1992) no. 4, 653--674.

    \bibitem{Kobayashi2}
    T. Kobayashi,
    {\it Casson-Gordon's rectangle condition of Heegaard diagrams
    and incompressible tori in $3$-manifolds},
    Osaka J. Math. {\bf 25} (1988) no. 3, 553-573.

    \bibitem{Kobayashi3}
    T. Kobayashi,
    {\it Heegaard splittings of exteriors of two bridge knots},
    Geom. and Topol. {\bf 5} (2001) 609--650.

    \bibitem{Lee}
    J. H. Lee,
    {\it Parity condition for irreducibility of Heegaard splittings},
    preprint, arXiv:0812.0225.

    \bibitem{MS}
    Y. Moriah \and J. Schultens,
    {\it Irreducible Heegaard splittings of Seifert fibered spaces
    are either vertical or horizontal},
    Topology {\bf 37} (1998) no. 5, 1089--1112.

    \bibitem{Saito}
    T. Saito,
    {\it Disjoint pairs of annuli and disks for Heegaard
    splittings},
    J. Korean Math. Soc. {\bf 42} (2005) no 4, 773--793.

    \bibitem{Sedgwick}
    E. Sedgwick,
    {\it The irreducibility of Heegaard splittings of Seifert fibered
    spaces},
    Pacific J. Math. {\bf 190} (1999) no 1, 173--199.

    \bibitem{Schultens}
    J. Schultens,
    {\it The classfication of Heegaard splittings for (compact
    orientable surface) $\times S^1$},
    Proc. London Math. Soc. (3) {\bf 67} (1993) no. 2, 425--448.
\end{thebibliography}
\end{document}